\documentclass[12pt,oneside,english]{amsart}
\usepackage[T1]{fontenc}
\usepackage[latin1]{inputenc}
\usepackage{verbatim}
\usepackage{amsthm}
\usepackage{amssymb}
\usepackage{esint}

\advance \topmargin by -\headheight
\advance \topmargin by -\headsep
\evensidemargin \oddsidemargin
\makeatletter
\numberwithin{equation}{section}
\numberwithin{figure}{section}
\theoremstyle{plain}
\newtheorem{thm}{Theorem}[section]
  \theoremstyle{definition}
  \newtheorem{defn}[thm]{Definition}
  \theoremstyle{plain}
  \newtheorem{lem}[thm]{Lemma}

\usepackage{epsfig}\usepackage{mathrsfs}\usepackage{babel}\numberwithin{equation}{section}

\date{\today}

\newtheorem{lemma}{Lemma}[section]\newtheorem{theorem}[lemma]{Theorem}\theoremstyle{theoreml}\theoremstyle{proposition}\theoremstyle{corollary}\theoremstyle{definition}\newtheorem{definition}[lemma]{Definition}\newtheorem{remark}[lemma]{Remark}

\newcommand{\beq}{\begin{equation*}}\newcommand{\eeq}{\end{equation*}}\newcommand{\be}{\begin{equation}}\newcommand{\ee}{\end{equation}}\newcommand{\bp}{\begin{proof}}\newcommand{\ep}{\end{proof}}\newcommand{\bd}{\begin{definition}}\newcommand{\ed}{\end{definition}}\newcommand{\bt}{\begin{theorem}}\newcommand{\et}{\end{theorem}}
\newcommand{\Z}{\mathbb{Z}}\newcommand{\B}{\mathscr{B}}\newcommand{\ra}{\rightarrow}

\def\Aut{{\textrm{Aut}}}

\def\cB{{\mathcal B}}

\def\cc{\curvearrowright}

\def\E{{\mathbb E}}

\def\ed{\textrm{End}}
\def\Ex{{\rm{Ex}}}

\def\cF{{\mathcal F}}

\def\oF{{\overline{F}}}

\def\G{{\rm{G}}}

\def\cP{{\mathcal P}}

\def\cQ{{\mathcal{Q}}}

\def\cR{{\mathcal{R}}}

\def\bt{{\bf t}}

\def\chix{{\raise.5ex\hbox{$\chi$}}}

\def\Z{{\mathbb Z}}

\makeatother

\usepackage{babel}

\makeatother

\usepackage{babel}

\makeatother

\usepackage{babel}

\begin{document}

\author{Lewis Bowen \& Yonatan Gutman}

\title{A Juzvinski{\u{\i}} Addition Theorem for Finitely Generated Free Groups Actions}
\begin{abstract}
The classical \textit{Juzvinski{\u{\i}} Addition Theorem} states
that the entropy of an automorphism of a compact group decomposes
along invariant subgroups. Thomas generalized the theorem to a skew-product
setting. Using L. Bowen's \textit{f-invariant }we prove the addition
theorem for actions of finitely generated free groups on skew-products
with compact totally disconnected groups or compact Lie groups (correcting an error from \cite{Bo10c})
and discuss examples. 

\end{abstract}


\maketitle

\tableofcontents{}

\section{Introduction}

The following result was proven independently by H. Li \cite{Li11}
and Lind-Schmidt \cite{LS09}.
\begin{thm}
\label{thm:amenable}{[}Addition theorem for amenable groups{]} Let
$\Gamma$ be a countable discrete amenable group, $G$ be a compact
metrizable group and $\alpha:\Gamma\to\Aut(G)$ an action of $\Gamma$
on $G$ by group-automorphisms. Suppose $N\lhd G$ is a closed normal
$\alpha(\Gamma)$-invariant subgroup. Denote by $\alpha_{N}:\Gamma\to\Aut(N)$
and $\alpha_{G/N}:\Gamma\to\Aut(G/N)$ the induced actions and by
$\mu_{G},\mu_{N},\mu_{G/N}$ the Haar probability measures on $G,N$
and $G/N$ respectively. Then the entropies of these actions satisfy:
\[
h_{\mu_{G}}(\alpha)=h_{\mu_{N}}(\alpha_{N})+h_{\mu_{G/N}}(\alpha_{G/N}).\]

\end{thm}
In the case $\Gamma=\Z$, this result is due to Juzvinski{\u{\i}}
\cite{Ju65} from which it receives its name. The case $\Gamma=\Z^{d}$
was proven in \cite{LSW90}. Special cases were obtained by Miles
\cite{Mi08} and Björklund-Miles \cite{BM09}.

The paper \cite{Bo10a} introduced a measure-conjugacy invariant,
called the {\em $f$-invariant}, for probability-measure-preserving
actions of finitely generated free groups. (Later a more general theory
of sofic entropy was introduced in \cite{Bo10b}, of which we have
little to say in the present article). In \cite{Bo10c}, a proof is
claimed that the above addition formula extends to the case when $\Gamma$
is a finitely generated free group, the entropy is replaced with the
$f$-invariant, and $G$ is either totally disconnected, a Lie group,
or a connected  abelian group (whenever the $f$-invariant
is well-defined). However, there is an error in the proof. We prove here that the statement remains correct if either $G$ is totally disconnected (and a mild additional hypothesis is satisfied) or $G$ is a  Lie group and the action is by smooth automorphisms. See the corrigendum \cite{BG12} for other corrections to \cite{Bo10c}. The main
result is Theorem \ref{thm:main} below. We also prove a skew-product
addition formula in Theorem \ref{thm:main-skew} which may be of independent
interest.

\subsection*{Organization} 
 \S \ref{S-f-invariant} reviews the $f$-invariant and states the
main theorem; \S \ref{sec:skew} reviews skew-products and proves
Theorem \ref{thm:main} from Theorem \ref{thm:main-skew}. In \S
\ref{sec:key} and \S \ref{sec:proof} Theorem \ref{thm:main-skew}
is proven; \S \ref{sec:example} discusses examples, including the
Ornstein-Weiss example. 

\subsection*{Acknowledgements} We would like to thank Hanfeng Li for helpful
conversations and the Fields Institute where some of the work for
this project occurred. Y.G. would like to thank Benjamin
Weiss, Eli Glasner and Jon Aaronson for helpful discussions. L.B.
was partially supported by NSF grants DMS-0968762 and DMS-0954606.




\section{The $f$-invariant}

\label{S-f-invariant} Let $\Gamma=\langle s_{1},\ldots,s_{r}\rangle$
be a rank $r$ free group. Let $\alpha$ be a measure-preserving action
of $\Gamma$ on a standard probability space $(X,\cB_{X},\nu)$. We
consider $\alpha$ as a homomorphism from $\Gamma$ to the group of
automorphisms of $(X,\cB_{X},\nu)$ and write $\alpha_{g}$ for $\alpha(g)$
( $\forall g\in\Gamma$). Let $\cP=\{P_{1},P_{2},\ldots\}$ be a countable
partition of $X$ into measurable subsets. The Shannon-entropy of
$\cP$ is \[
H_{\nu}(\cP)\triangleq -\sum_{P\in\cP}\nu(P)\log(\nu(P)).\]
 By convention $0\log(0)\triangleq 0$. If $\cP,\cQ$ are two partitions of
$X$ then their {\em join} is defined by $\cP\vee\cQ\triangleq\{P\cap Q:~P\in\cP,Q\in\cQ\}$.
If $W\subset\Gamma$ is finite, we let $\cP^{W}\triangleq\bigvee_{w\in W}\alpha_{w}\cP$.
Note that $\alpha$ is only implicit in this notation. If $H_\nu(\cP)<\infty$
then define \begin{eqnarray*}
F_{\nu}(\alpha,\cP) & = & (1-2r)H_{\nu}(\cP)+\sum_{i=1}^{r}H_{\nu}(\cP\vee\alpha_{s_{i}}\cP)\\
f_{\nu}(\alpha,\cP) & = & \inf_{n>0}F_{\nu}\left(\alpha,\cP^{B(n)}\right)\end{eqnarray*}
 where $B(n)$ is the ball of radius $n$ centered at the identity with respect to the word metric. The partition $\cP$ is said to be {\em generating} (for the action
$\alpha$) if the smallest $\alpha(\Gamma)$-invariant $\sigma$-algebra
containing $\cP$ is $\cB_{X}$ (up to sets of measure zero). In \cite{Bo10a} it is shown that if $\cP,\cQ$ are finite-entropy generating partitions then $f_\nu(\alpha,\cP)=f_\nu(\alpha,\cQ)$. So we define the $f$-invariant of the action by $f_{\nu}(\alpha)\triangleq f_{\nu}(\alpha,\cP)$
where $\cP$ is any finite-entropy generating partition for $\alpha$.
If there does not exist a finite-entropy generating partition for
$\alpha$ then $f_{\nu}(\alpha)$ is undefined.

It will be useful to have an alternative formulation of the $f$-invariant for which we need the following definitions. For $g\in\Gamma$, let $h_{\nu}(\alpha_{g},\cP)$ denote the entropy
rate of $\cP$ with respect to the $\Z$-action generated by $\alpha_{g}$.
To be precise, \[
h_{\nu}(\alpha_{g},\cP)=\lim_{n\to\infty}\frac{1}{2n+1}H\left(\bigvee_{i=-n}^{n}\alpha_{g}^{i}\cP\right).\]
 The {\em entropy} of the action $\alpha_{g}$ is $h_{\nu}(\alpha_{g})=\sup_{\cP}h_{\nu}(\alpha_{g},\cP)$
where the supremum is over all finite measurable partitions $\cP$
of $X$. Define \begin{eqnarray*}
F_{\nu}^{*}(\alpha,\cP) & = & (1-r)H_{\nu}(\cP)+\sum_{i=1}^{r}h_{\nu}(\alpha_{s_{i}},\cP)\\
f_{\nu}^{*}(\alpha,\cP) & = & \inf_{n>0}F_{\nu}^{*}\left(\alpha,\cP^{B(n)}\right).\end{eqnarray*}

One of the main results
of \cite{Bo10c} is:
\begin{thm}
\label{thm:f*} Let $\alpha$ be a measure-preserving action of $\Gamma$
on a standard probability space $(X,\cB_{X},\nu)$. Then for any finite-entropy
generating partition $\cP$ for $\alpha$, $f_{\nu}(\alpha)=f_{\nu}^{*}(\alpha,\cP)$.
\end{thm}
The main result of this paper is:
\begin{thm}
\label{thm:main} Let $\Gamma=\langle s_{1},\ldots,s_{r}\rangle$
be a rank $r$ free group, $G$ be a compact metrizable group and
$\alpha:\Gamma\to\Aut(G)$ an action of $\Gamma$ on $G$ by group-automorphisms.
Suppose $N\lhd G$ is a closed normal $\alpha(\Gamma)$-invariant
subgroup. Denote by $\alpha_{N}:\Gamma\to\Aut(N)$ and $\alpha_{G/N}:\Gamma\to\Aut(G/N)$
the induced actions and by $\mu_{G},\mu_{N},\mu_{G/N}$ the Haar probability
measures on $G,N$ and $G/N$ respectively. Suppose there exists finite-entropy
generating partitions for $\alpha,\alpha_{N},\alpha_{G/N}$ and one
of the following hold.
\begin{enumerate}
\item $N$ is totally disconnected and there exists a clopen finite-index
normal subgroup $N_{0}\lhd N$ such that $\{gN_{0}:~g\in N\}$ is
a generating partition for $\alpha_{N}$.
\item $G$ is a compact  Lie group and the action $\alpha$
is by smooth automorphisms.
\end{enumerate}
Then \[
f_{\mu_{G}}(\alpha)=f_{\mu_{N}}(\alpha_{N})+f_{\mu_{G/N}}(\alpha_{G/N}).\]

\end{thm}
\begin{remark} The proof shows slightly more: if case (1) occurs
and $\alpha_{G/N}$ has a finite-entropy generating partition, then
$\alpha$ automatically has a finite-entropy generating partition.
This follows from Lemmas \ref{lem:standard} and \ref{lem:generating2}.
To be more precise, Lemma \ref{lem:standard} shows that $\alpha$
is measurably conjugate to a skew-product action of the form $\alpha_{G/N}\times_{\sigma}\alpha_{N}$.
If $\cP$ is a finite-entropy generating partition for $\alpha_{G/N}$
and $\cQ=\{gN_{0}:~g\in G\}$ is a generating partition for $\alpha_{N}$
of the kind described in case (1) above, then Lemma \ref{lem:generating2}
shows that $\cP\times\cQ$ is generating for $\alpha_{G/N}\times_{\sigma}\alpha_{N}$.
Because $\cP$ has finite-entropy and $\cQ$ is finite, $\cP\times\cQ$
has finite entropy as required. \end{remark}

\begin{remark} Suppose $N$ as above is totally disconnected and $N_{0}\lhd N$
is a closed finite-index normal subgroup (the fact $N_{0}$ is closed
and finite-index implies $N_{0}$ is clopen). Let $M=\bigcap_{g\in\Gamma}\alpha_{g}N_{0}$. Note $M$ is $\alpha(\Gamma)$-invariant. Let $\alpha_{G/M},\alpha_{N/M}$ be the induced actions on $G/M$
and $N/M$ respectively. Let $\mu_{G/M},\mu_{N/M}$ be the respective
Haar probability measures. Suppose that $\alpha_{G/M}$ and $\alpha_{G/N}$
admit finite-entropy generating partitions. Note that the clopen partition $\{gN_{0}/M:~g\in~G\}$
is a finite generating partition for $\alpha_{N/M}$, as it is the image of the clopen generating partion $\{gN_{0}:~g\in~G\}$ under the continuous projection $\pi:N\rightarrow N/M$. So the theorem above
implies \[
f_{\mu_{G/M}}(\alpha_{G/M})=f_{\mu_{N/M}}(\alpha_{N/M})+f_{\mu_{G/N}}(\alpha_{G/N}).\]
 By the previous remark, this formula holds as long as $\alpha_{G/N}$
admits a finite-entropy generating partition. \end{remark}

\section{Skew-products}

\label{sec:skew}

The proof of Theorem \ref{thm:main} is based on a more general skew-product
theorem of independent interest, the construction of which we recall
next.
\begin{defn}
\label{def:skew-product} Let $\Gamma$ be a group. Let $(X,\cB_{X},\nu)$
be a Lebesgue space equipped with a $\Gamma$-action $\alpha$ and $\G$ be a compact group with Borel $\sigma$-algebra $\B$ and Haar
measure $\mu$. Let $\beta$ be a $\Gamma$-action by group-automorphisms
on $\G$. Let $\sigma:\Gamma\times X\ra\G$ be a cocycle for $\beta$
and $\alpha$, i.e., $\sigma$ is a measurable mapping so that for
all $g,h\in\Gamma$, $x\in X$ \begin{equation}
\sigma(gh,x)=(\beta_{g}\sigma(h,x))\cdot\sigma(g,\alpha_{h}x).\label{eq:cocycle}\end{equation}
 Define the \em{skew-product action} $\alpha\times_{\sigma}\beta$
of $\Gamma$ on $X\times\G$ by: \[
(\alpha\times_{\sigma}\beta)_{g}(x,y)=(\alpha_{g}x,(\beta_{g}y)\cdot\sigma(g,x))\,\,(g\in\Gamma,\, x\in X,y\in\G).\]

\end{defn}
The connection between skew-product actions and the addition theorem
is the following standard result (which we obtained from \cite[Proof of Corollary 6.3]{Li11}).
\begin{lem}
\label{lem:standard} Let $\Gamma$ be a countable group, $G$ be
a compact metrizable group, $\alpha:\Gamma\to\Aut(G)$ an action of
$\Gamma$ on $G$ by group-automorphisms and $N\lhd G$ a closed normal
$\alpha(\Gamma)$-invariant subgroup. Denote by $\alpha_{N}:\Gamma\to\Aut(N)$
and $\alpha_{G/N}:\Gamma\to\Aut(G/N)$ the induced actions. Then there
is a cocycle $\sigma:\Gamma\times G/N\to N$ such that $\alpha_{G/N}\times_{\sigma}\alpha_{N}$
is measurably conjugate with $\alpha$.
\end{lem}
The main technical result of this paper is:
\begin{thm}
\label{thm:main-skew} Let $\Gamma=\langle s_{1},\ldots,s_{r}\rangle$
be a rank $r$ free group, $\alpha$ a measure-preserving action of
$\Gamma$ on a standard probability space $(X,\cB_{X},\nu)$, $G$
a compact metrizable group, $\beta$ an action of $\Gamma$ on $G$
by group-automorphisms, and $\sigma:\Gamma\times X\to G$ a cocycle
for these actions. Suppose that $G$ is totally disconnected and there
exists a finite-index clopen normal subgroup $N\lhd G$ such that
$\{gN:~g\in G\}$ is a generating partition for $\beta$. Let $\mu$
denote the Haar probability measure on $G$. Suppose also that there
is a finite-entropy generating partition for $\alpha$. Then \[
f_{\nu\times\mu}(\alpha\times_{\sigma}\beta)=f_{\nu}(\alpha)+f_{\mu}(\beta).\]

\end{thm}
The analog of this theorem for discrete countable amenable groups
$\Gamma$ when $G$ is an arbitrary compact metrizable group was established
in \cite{Li11}. The case $\Gamma=\Z$ was proven earlier by Thomas
\cite{Th71} and the case $\Gamma=\Z^{d}$ is shown in \cite{LSW90}.

Theorem \ref{thm:main-skew} is proven in the next section. Next we
combine this result with the following two lemmas to complete the proof
of Theorem \ref{thm:main}.
\begin{lem}
Let $M$ be a smooth compact Riemannian manifold. Let $T:M\to M$
be a diffeomorphism. Then $h_{\mu}(T)<\infty$ for any $T$-invariant
probability measure $\mu$. \end{lem}
\begin{proof}
This is due to Kushnirenko \cite{Ku65}. Alternatively, it follows
from Ruelle's inequality (see e.g. \cite[Corollary S.2.17]{KH95}). \end{proof}
\begin{lem}
Let $\Gamma=\langle s_{1},\ldots,s_{r}\rangle$ be a rank $r$ free
group with $r>1$, $M$ be a smooth compact Riemannian manifold, $\alpha$
a measure-preserving action of $\Gamma$ on $M$ by diffeomorphisms
and $\mu$ a non-atomic $\alpha(\Gamma)$-invariant probability measure
on $M$. Then $f_{\mu}(\alpha)=-\infty$ if there is a finite-entropy
generating partition for the action.
\end{lem}
\begin{proof}
Let $m=\max_{i=1}^{r}h_{\mu}(\alpha_{s_{i}})$. By the previous lemma,
$m<\infty$. Let $\cP$ be a finite-entropy generating partition for
$\alpha$. Let $N>0$. Because $\mu$ is non-atomic, there is a finite
partition $\cQ$ of $M$ with $H_{\mu}(\cQ)>N$. So after replacing
$\cP$ with $\cP\vee\cQ$ if necessary, we may assume that $H_{\mu}(\cP)>N$.
By Theorem \ref{thm:f*} \begin{eqnarray*}
f_{\mu}(\alpha) & = & f_{\mu}^{*}(\alpha,\cP)=\inf_{n>0}F_{\mu}^{*}(\alpha,\cP^{B(n)})\\
 & \le & (1-r)H_{\mu}(\cP)+\sum_{i=1}^{r}h_{\mu}(\alpha_{s_{i}},\cP)\\
 & \le & (1-r)N+rm.\end{eqnarray*}
 Because $N>0$ is arbitrary and $r>1$, this implies the lemma.
\end{proof}

\begin{proof}[Proof of Theorem \ref{thm:main} from Theorem \ref{thm:main-skew}]

Suppose item (1) holds. By Lemma \ref{lem:standard}, $\alpha$ is
measurable conjugate with $\alpha_{G/N}\times_{\sigma}\alpha_{N}$
for some cocycle $\sigma$. So Theorem \ref{thm:main-skew} implies
\[
f_{\mu_{G}}(\alpha)=f_{\mu_{G}}(\alpha_{G/N}\times_{\sigma}\alpha_{N})=f_{\mu_{G/N}}(\alpha_{G/N})+f_{\mu_{N}}(\alpha_{N})\]
 as required.

Suppose that item (2) holds; i.e., $G$ is a  compact
Lie group and $\alpha$ is an action by smooth group-automorphisms. If $G$ is finite then the theorem is clear because
\begin{eqnarray*}
f_{\mu_{G}}(\alpha)&=&-(r-1)\log|G|=-(r-1)\log|G/N|-(r-1)\log|N|\\
&=&f_{\mu_{G/N}}(\alpha_{G/N})+f_{\mu_{N}}(\alpha_{N}).
\end{eqnarray*}


By Theorem \ref{thm:amenable}, we may assume, without loss of generality,
that $r>1$. If $G$ is infinite then, because it is compact, it has positive
dimension. So $\mu_{G}$ is non-atomic and the previous lemma implies
$f_{\mu_{G}}(\alpha)=-\infty$. Also, because $G$ is infinite, either $N$ or $G/N$ is infinite.
Therefore, either $\mu_{N}$ or $\mu_{G/N}$ is non-atomic. Of course,
the actions $\alpha_{N}$ and $\alpha_{G/N}$ are smooth (because every continuous homomorphism between Lie groups is analytic \cite[Ch. II, Theorem 2.6]{He01}). It should
be noted that the $f$-invariant does not take on the value $+\infty$.
So the previous lemma implies $f_{\mu_{G/N}}(\alpha_{G/N})+f_{\mu_{N}}(\alpha_{N})=-\infty$.
\end{proof}

\section{Relative entropy}

\label{sec:relative} The proof of Theorem \ref{thm:main-skew} uses
the relative $f$-invariant theory developed in \cite{Bo10c}, which
we review here. Let $(X,\cB_{X},\nu)$ be a standard probability
space. Let $\cP$ be a countable measurable partition of $X$ and
let $\cF\subset\cB_{X}$ be a sub-sigma algebra. Recall that for a.e.
$x\in X$, the conditional expectation $\E[\cdot|\cF](x)$ is a probability
measure on $(X,\cB_{X})$ satisfying
\begin{enumerate}
\item $x\mapsto\E[A|\cF](x)$ is $\cF$-measurable for any $A\in\cB_{X}$;
\item $\int\E[A|\cF](x)~d\nu(x)=\nu(A)$ for any $A\in\cB_{X}$.
\end{enumerate}
The information function $I(\cP|\cF)$ is a function on $X$ defined
by \[
I(\cP|\cF)(x)=-\E[P_{x}|\cF](x)\log(\E[P_{x}|\cF](x))\]
 where $P_{x}\in\cP$ is the unique partition element with $x\in P_{x}$.
The Shannon entropy of $\cP$ {\em relative} to $\cF$ is \[
H_{\nu}(\cP|\cF)=\int I(\cP|\cF)(x)~d\nu(x).\]
 If $T$ is a measure-preserving transformation of $(X,\cB_{X},\nu)$
then the entropy rate of $(T,\cP)$ {\em relative} to $\cF$ is
\[
h_{\nu}(T,\cP|\cF)=\lim_{n\to\infty}\frac{1}{2n+1}H_{\nu}\left(\bigvee_{i=-n}^{n}T^{i}\cP|~\cF\right).\]
 This is well-defined whenever $\cF$ is $T$-invariant. We also define
the entropy rate of $T$ {\em relative} to $\cF$ by \[
h_{\nu}(T|\cF)=\sup_{\cP}h_{\nu}(T,\cP|\cF)\]
 where the supremum is over all finite-entropy partitions $\cP$ of
$X$.

Now suppose $\Gamma=\langle s_{1},\ldots,s_{r}\rangle$ and $\alpha$
is a measure-preserving action of $\Gamma$ on $(X,\cB_{X},\nu)$.
Define \begin{eqnarray*}
F_{\nu}(\alpha,\cP|\cF) & = & (1-2r)H_{\nu}(\cP|\cF)+\sum_{i=1}^{r}H_{\nu}(\cP\vee\alpha_{s_{i}}\cP|\cF)\\
f_{\nu}(\alpha,\cP|\cF) & = & \inf_{n>0}F_{\nu}\left(\alpha,\cP^{B(n)}|\cF\right).\end{eqnarray*}
Also define
\begin{eqnarray*}
F_{\nu}^{*}(\alpha,\cP|\cF) & = & (1-r)H_{\nu}(\cP|\cF)+\sum_{i=1}^{r}h_{\nu}(\alpha_{s_{i}},\cP|\cF)\\
f_{\nu}^{*}(\alpha,\cP|\cF) & = & \inf_{n>0}F_{\nu}^{*}\left(\alpha,\cP^{B(n)}|\cF\right).\end{eqnarray*}
\begin{thm}
\label{thm:KS-rel} Let $\alpha$ be a measure-preserving action of
$\Gamma$ on a standard probability space $(X,\cB_{X},\nu)$. If $\cP,\cQ$
are any two finite-entropy generating partitions for $\alpha$ and
$\cF\subset\cB_{X}$ is an $\alpha(\Gamma)$-invariant sub-$\sigma$-algebra
then $f_{\nu}(\alpha,\cP|\cF)=f_{\nu}(\alpha,\cQ|\cF)=f_{\nu}^*(\alpha,\cP|\cF)=f_{\nu}^*(\alpha,\cQ|\cF)$. \end{thm}
\begin{proof}
This is implied by \cite[Theorems 5.3, 9.1]{Bo10c}. The proof requires a small correction;
see \cite{BG12}.\end{proof}
Because of this theorem, we define $f_{\nu}(\alpha|\cF)\triangleq f_{\nu}(\alpha,\cP|\cF)$
where $\cP$ is any finite-entropy generating partition for $\alpha$.
If there does not exist a finite-entropy generating partition for
$\alpha$ then $f_{\nu}(\alpha|\cF)$ is undefined.
\begin{thm}
\label{thm:AR}{[}The $f$-invariant Abramov-Rokhlin Addition Formula{]}
Let $\alpha$ be a measure-preserving action of $\Gamma$ on a standard
probability space $(X,\cB_{X},\nu)$. Let $\cP,\cQ$ be finite-entropy
partitions of $X$. Let $\Sigma(\cQ)$ be the smallest $\Gamma$-invariant
sub-$\sigma$-algebra containing $\cQ$. Then \[
f_{\nu}(\alpha,\cP\vee\cQ)=f_{\nu}(\alpha,\cQ)+f_{\nu}(\alpha,\cP|\Sigma(\cQ)).\]
\end{thm}
\begin{proof}
This is \cite[Theorem 1.3]{Bo10c}. The proof requires a small correction;
see \cite{BG12}.
\end{proof}

\section{A key Lemma}

\label{sec:key}

The purpose of this section is to prove the key lemma below for skew-products
of $\Z$-actions. Let $(X,\cB_{X},\nu)$ be a Lebesgue space, $T\in\Aut(X,\cB_{X},\nu)$,
$\G$ a compact metrizable group, equipped with Haar measure $\mu$
and $S$ a group-automorphism of $\G$. A cocycle for $T$ and
$S$ is a cocycle for the actions of $\Z$ induced by $T$ and $S$.
That is, it is a measurable map $\sigma:\Z\times X\rightarrow\G$
such that \begin{equation}
\sigma(n+m,x)=(S^{n}\sigma(m,x))\cdot\sigma(n,T^{m}x).\label{eq:cocycle2}\end{equation}

\begin{lem}
\label{Lemma:Skew Entropy for clopen partition} Let $(X,\cB_{X},\nu),\G,T,S,\sigma$
be as above. Let $\cQ$ be a finite-entropy partition of $\G$. Let
\[
K(\cQ)=\sup_{g\in\G}H(\cQ g|\cQ)+H(\cQ|\cQ g).\]
 Then 
\[
\Big|h_{\nu\times\mu}\big(T\times_{\sigma}S,X\times\cQ|\cB_{X}\big)-h_{\mu}(S,\cQ)\Big|\le K(\cQ).\]
 \end{lem}
\begin{proof}
By the definition of conditional entropy :

\begin{eqnarray*}
&&h_{\nu\times\mu}(T\times_{\sigma}S,X\times\cQ|\cB_{X})\\
&=&\lim_{m\rightarrow\infty}\frac{1}{m}H_{\nu\times\mu}\left(\bigvee_{k=0}^{m-1}(T\times_{\sigma}S)^{-k}X\times\cQ|\cB_{X}\right)\\
&=&\lim_{m\rightarrow\infty}\frac{1}{m}\int H_{\delta_x\times\mu}\left(\bigvee_{k=0}^{m-1}(T\times_{\sigma}S)^{-k}X\times\cQ\right) ~d\nu(x)
\end{eqnarray*}
where $\delta_x$ is the Dirac probability measure concentrated on $\{x\}$.


We claim that for any set $P\subset\G$, \[
\{x\}\times\G\cap(T\times_{\sigma}S)^{-k}(X\times P)=\{x\}\times S^{-k}(P\sigma(k,x)^{-1}).\]
 Indeed, $(x,y)$ is contained in $(T\times_{\sigma}S)^{-k}(X\times P)$
if and only if \[
(T\times_{\sigma}S)^{k}(x,y)=(T^{k}x,(S^{k}y)\sigma(k,x))\in X\times P\]
 which occurs if and only if \[
y\in S^{-k}(P\sigma(k,x)^{-1}).\]
So if \[
\cQ_{x}^{m}=\bigvee_{k=0}^{m-1}S^{-k}(\cQ\sigma(k,x)^{-1}).\]
 then \begin{eqnarray*}
H_{\delta_x\times\mu}\left(\bigvee_{k=0}^{m-1}(T\times_{\sigma}S)^{-k}X\times\cQ\right) & = & H_{\mu}(\cQ_{x}^{m})\end{eqnarray*}
which implies: \begin{equation}
h_{\nu\times\mu}((T\times_{\sigma}S),X\times\cQ|\cB_{X})=\lim_{m\ra\infty}\frac{1}{m}\int_{X}H_{\mu}(\cQ_{x}^{m})d\nu(x)\label{skewed entropy}\end{equation}
Define: \[
\cQ^{m}=\bigvee_{k=0}^{m-1}S^{-k}\cQ\]
 By the definition of entropy: \begin{equation}
h_{\mu}(S,\cQ)=\lim_{m\ra\infty}\frac{1}{m}\int_{X}H_{\mu}(\cQ^{m})d\nu(x)\label{direct entropy}\end{equation}
Note $|H_{\mu}(\cQ^{m})-H_{\mu}(\cQ_{x}^{m})|\leq H_{\mu}(\cQ^{m}|\cQ_{x}^{m})+H_{\mu}(\cQ_{x}^{m}|\cQ^{m})$.
Thus:

\begin{eqnarray*}
|H_{\mu}(\cQ^{m})-H_{\mu}(\cQ_{x}^{m})| & \leq & \sum_{k=0}^{m-1}H_{\mu}(S^{-k}\cQ|S^{-k}(\cQ\sigma(k,x)^{-1}))+H_{\mu}(S^{-k}(\cQ\sigma(k,x)^{-1})|S^{-k}\cQ)\\
 & = & \sum_{k=0}^{m-1}H_{\mu}(\cQ|\cQ\sigma(k,x)^{-1})+H_{\mu}(\cQ\sigma(k,x)^{-1})|\cQ)\le mK(\cQ).\end{eqnarray*}
Finally (\ref{skewed entropy}) and (\ref{direct entropy}) imply
$\left|h_{\nu\times\mu}((T\times_{\sigma}S),X\times\cQ|\cB_{X})-h_{\mu}(S,\cQ)\right|\le K(\cQ)$.
\end{proof}

\section{Proof of Theorem \ref{thm:main-skew}}

\label{sec:proof}

For this section, let $\Gamma,(X,\cB_{X},\nu),(\G,\cB_{G},\mu),\alpha,\beta,\sigma$
be as in Theorem \ref{thm:main-skew}. A {\em special partition} of
$\G$ is a partition $\cQ$ such that there exists a finite-index
normal clopen subgroup $N<\G$ such that $\cQ=\{gN:~g\in\G\}$. The next lemma is left as an exercise for the reader.
\begin{lem}
If $\cQ$ is special and $T_{1},\ldots,T_{n}$ are automorphisms of
$\G$ then $\bigvee_{i=1}^{n}T_{i}\cQ$ is also special. \end{lem}

\begin{lem}
\label{lem:key} If $\cP$ is any finite-entropy partition of $X$
and $\cQ$ is a special partition of $G$ then \[
F_{\nu\times\mu}^{*}(\alpha\times_{\sigma}\beta,\cP\times\cQ|\cB_{X})=F_{\mu}^{*}(\beta,\cQ).\]
 \end{lem}
\begin{proof}
Because $\cQ g=\cQ$ for any $g\in G$, it follows that $K(\cQ)=0$
where $K(\cdot)$ is as defined in Lemma \ref{Lemma:Skew Entropy for clopen partition}.
So that Lemma implies \begin{eqnarray*}
F_{\nu\times\mu}^{*}(\alpha\times_{\sigma}\beta,\cP\times\cQ|\cB_{X}) & = & (1-r)H_{\nu\times\mu}(\cP\times\cQ|\cB_{X})+\sum_{i=1}^{r}h_{\nu\times\mu}((\alpha\times_{\sigma}\beta)_{s_{i}},\cP\times\cQ|\cB_{X})\\
 & = & (1-r)H_{\mu}(\cQ)+\sum_{i=1}^{r}h_{\mu}(\beta_{s_{i}},\cQ)=F_{\mu}^{*}(\beta,\cQ).\end{eqnarray*}

\end{proof}


\begin{lem}
\label{lem:keyequation} Let $\cQ$ be a special partition of $G$,
$g\in\Gamma$ and $\cP_{g}$ denote the partition of $X$ obtained
by pulling $\beta_{g}(\cQ)$ back under the cocycle $\sigma(g,\cdot)$.
Also, let $\cP'$ be an arbitrary measurable partition of $X$. Then
\[
(\alpha\times_{\sigma}\beta)_{g}((\cP_{g}\vee\cP')\times\cQ)=\alpha_{g}(\cP_{g}\vee\cP')\times\beta_{g}(\cQ)\]
 (up to sets of measure zero). \end{lem}
\begin{proof}
Let $N$ be the finite-index clopen normal subgroup of $G$ such that
$\cQ=\{qN:~q\in G\}$. Let $P\in\cP_{g},P'\in\cP'$ and $qN\in\cQ$.
It suffices to show that there exists some $q''\in G$ such that \[
(\alpha\times_{\sigma}\beta)_{g}((P\cap P')\times qN)=\alpha_{g}(P\cap P')\times q''\beta_{g}(N)\]
 up to sets of measure zero. By definition of $\cP_{g}$, there exists
a coset $q'\beta_{g}(N)\in G/\beta_{g}(N)$ such that for every $y\in P$,
$\sigma(g,y)\in q'\beta_{g}(N)$.

Let $x\in P\cap P'$ and $n\in N$. Then there exists some $m\in N$
such that \[
(\alpha\times_{\sigma}\beta)_{g}(x,qn)=(\alpha_{g}x,\beta_{g}(qn)\sigma(g,x))=(\alpha_{g}x,\beta_{g}(qn)q'\beta_{g}(m)).\]
 Because $N$ is normal, $\beta_{g}(qn)q'\beta_{g}(m)\in\beta_{g}(q)q'\beta_{g}(N)$.
Thus $(\alpha\times_{\sigma}\beta)_{g}(x,qn)\in\alpha_{g}(P\cap P')\times\beta_{g}(q)q'\beta_{g}(N)$.
Since $(\alpha\times_\sigma \beta)_g$ preserves $\nu\times \mu$,  it follows that $(\alpha\times_{\sigma}\beta)_{g}((P\cap P')\times qN)=\alpha_{g}(P\cap P')\times q''\beta_{g}(N)$
up to sets of measure zero. 
 \end{proof}
\begin{lem}
\label{lem:generating2} Let $\cP,\cQ$ be measurable partitions for
$\alpha,\beta$ respectively. Suppose $\cQ$ is special and $\cP$
is generating. Let $\Sigma(\cP,\cQ)$ be the smallest $\alpha\times_{\sigma}\beta(\Gamma)$-invariant
$\sigma$-algebra containing $\cP\times\cQ$. Similarly, let $\Sigma(\cQ)$
be the smallest $\beta(\Gamma)$-invariant $\sigma$-subalgebra of $\cB_{G}$
which contains $\cQ$.

Then $\Sigma(\cP,\cQ)$ is the smallest $\sigma$-algebra containing $\cB_{X}\times\Sigma(\cQ)$
(up to sets of measure zero).
\end{lem}


\begin{proof}
Clearly, $\cP\times G$ is contained in $\Sigma(\cP,\cQ)$. Because
\[
(\alpha\times_{\sigma}\beta)_{g}(\cP\times\G)=(\alpha_{g}\cP)\times G,\quad\forall g\in\Gamma,\]
 it follows that $(\alpha_{g}\cP)\times G\subset\Sigma(\cP,\cQ)$
for every $g\in\Gamma$. Because $\cP$ is generating, this implies
$\cB_{X}\times G\subset\Sigma(\cP,\cQ)$ (up to sets of measure zero).

For each $g\in\Gamma$, recall that $\cP_{g}$ is the partition of
$X$ obtained by pulling $\beta_{g}(\cQ)$ back under the cocycle
$\sigma(g,\cdot)$. 
Because $\sigma(g,\cdot)$ is $\cB_{X}$-measurable, $\cP_{g}\times\cQ$
is contained in $\Sigma(\cP,\cQ)$. By Lemma \ref{lem:keyequation},
\[
(\alpha\times_{\sigma}\beta)_{g}(\cP_{g}\times\cQ)=(\alpha_{g}\cP_{g})\times(\beta_{g}\cQ)\subset\Sigma(\cP,\cQ)\]
 (up to sets of measure zero). Because $X\times\beta_{g}\cQ$ coarsens
$(\alpha_{g}\cP_{g})\times(\beta_{g}\cQ)$, it follows that $X\times\beta_{g}\cQ\subset\Sigma(\cP,\cQ)$
for every $g\in\Gamma$. By definition of $\Sigma(\cQ)$, this implies
$X\times\Sigma(\cQ)\subset\Sigma(\cP,\cQ)$. Because $X\times\Sigma(\cQ)$
and $\cB_{X}\times G$ generate $\cB_{X}\times\Sigma(\cQ)$ (up to
sets of measure zero), this implies $\Sigma(\cP,\cQ)\supset\cB_{X}\times\Sigma(\cQ)$.

To show the opposite inclusion, it suffices to show that $(\alpha\times_{\sigma}\beta_{g})(\cP\times\cQ)\in\cB_{X}\times\Sigma(\cQ)$
for any $g\in\Gamma$. By the previous lemma, \[
(\alpha\times_{\sigma}\beta_{g})(\cP\times\cQ)\le(\alpha\times_{\sigma}\beta)_{g}((\cP_{g}\vee\cP)\times\cQ)=(\alpha_{g}(\cP_{g}\vee\cP))\times(\beta_{g}\cQ)\in\cB_{X}\times\Sigma(\cQ).\]
 \end{proof}

\begin{proof}[Proof of Theorem \ref{thm:main-skew}]
 Let $\cP$ be a finite-entropy generating partition for $\alpha$
and $\cQ$ be a special generating partition for $\beta$. By the
previous lemma, $\cP\times\cQ$ is generating for $\alpha\times_{\sigma}\beta$.
So Theorems  \ref{thm:KS-rel} and \ref{thm:AR} imply \begin{eqnarray*}
f_{\nu\times\mu}(\alpha\times_{\sigma}\beta) - f_{\nu}(\alpha) &=& f_{\nu\times\mu}(\alpha\times_{\sigma}\beta|\cB_{X})\\
 & = & \inf_{n>0}F_{\nu\times\mu}^{*}(\alpha\times_{\sigma}\beta,(\cP\times\cQ)^{B(n)}|\cB_{X}).\end{eqnarray*}

For each $g\in\Gamma$, let $\cP_{g}$ be the partition of $X$ obtained
by pulling $(\beta_{g}\cQ)$ back under $\sigma(g,\cdot)$. By Lemma
\ref{lem:keyequation}, for any partition $\cP'$ of $X$ \[
(\alpha\times_{\sigma}\beta)_{g}((\cP_{g}\vee\cP')\times\cQ)=(\alpha_{g}(\cP_{g}\vee\cP'))\times(\beta_{g}\cQ).\]
Given an integer $n>0$ let $\cR_{n}=\bigvee_{g\in B(n)}\cP_{g}$. 
By Lemma \ref{lem:keyequation}, \begin{eqnarray*}
(\cP\vee\cR_{n})\times\cQ)^{B(n)} & = & \bigvee_{w\in B(n)}(\alpha\times_{\sigma}\beta)_{w}((\cP\vee\cR_{n})\times\cQ)\\
 & = & \bigvee_{w\in B(n)}\alpha_{w}(\cP\vee\cR_{n})\times\beta_{w}(\cQ)\\
 & = & (\cP\vee\cR_{n})^{B(n)}\times\cQ^{B(n)}.\end{eqnarray*}
 Because we are conditioning on $\cB_{X}$ and $(\cR_{n}\times G)^{B(n)}=(\cR_{n}^{B(n)}\times G)$,
\begin{eqnarray*}
F_{\nu\times\mu}^{*}(\alpha\times_{\sigma}\beta,(\cP\times\cQ)^{B(n)}|\cB_{X}) & = & F_{\nu\times\mu}^{*}(\alpha\times_{\sigma}\beta,(\cP\times\cQ)^{B(n)}\vee\cR_{n}^{B(n)}\times G|\cB_{X})\\
 & = & F_{\nu\times\mu}^{*}(\alpha\times_{\sigma}\beta,((\cP\vee\cR_{n})\times\cQ)^{B(n)}|\cB_{X})\\
 & = & F_{\nu\times\mu}^{*}(\alpha\times_{\sigma}\beta,((\cP\vee\cR_{n})^{B(n)}\times\cQ^{B(n)}|\cB_{X}).\end{eqnarray*}
By Lemma \ref{lem:key}, \[
F_{\nu\times\mu}^{*}(\alpha\times_{\sigma}\beta,((\cP\vee\cR_{n})^{B(n)}\times\cQ^{B(n)}|\cB_{X})=F_{\mu}^{*}(\beta,\cQ^{B(n)}).\]
 So we now have \begin{eqnarray*}
f_{\nu\times\mu}(\alpha\times_{\sigma}\beta) & = & f_{\nu}(\alpha)+f_{\nu\times\mu}(\alpha\times_{\sigma}\beta|\cB_{X})\\
 & = & f_{\nu}(\alpha)+\inf_{n>0}F_{\nu\times\mu}^{*}(\alpha\times_{\sigma}\beta,(\cP\times\cQ)^{B(n)}|\cB_{X})\\
 & = & f_{\nu}(\alpha)+\inf_{n>0}F_{\mu}^{*}(\beta,\cQ^{B(n)})\\
 & = & f_{\nu}(\alpha)+f_{\mu}(\beta).\end{eqnarray*}
 The last equality holds by Theorem \ref{thm:f*}.
\end{proof}


\section{Examples}

\label{sec:example}


It is convenient to introduce the following notation. Let $\Gamma=\langle s_{1},\ldots,s_{r}\rangle$
be the rank $r$ free group. If $K$ is a set then $K^{\Gamma}$ is
the set of all functions $x:\Gamma\to K$. The {\em shift-action}
of $\Gamma$ on $K^{\Gamma}$ is defined as follows. For $g,f\in\Gamma$
and $x\in K^{\Gamma}$, $gx\in K^{\Gamma}$ is the map $(gx)(f)=x(g^{-1}f)$.

If $\Gamma$ acts on a compact group $G$ and the action is understood,
we write $f(\Gamma\cc G)$ to mean the $f$-invariant of the action
of $G$ with respect to Haar measure.

\subsection{The Ornstein-Weiss Example}

This example comes from the appendix to \cite{OW87}. To explain its
relevance, let us recall some basic facts from classical entropy theory.
Let $\Delta$ be a countable amenable group, $K$ a finite set and $u$ the
uniform probability measure on $K$. It is straightforward to compute
the entropy of the shift action of $\Delta$ on $(K^{\Delta},u^{\Delta})$:
it is $\log|K|$. Because entropy never increases under a factor map,
it follows that if $|K|>1$ then the action $\Delta\cc(K^{\Delta},u^{\Delta})$
cannot factor onto the action $\Delta\cc((K\times K)^{\Delta},(u\times u)^{\Delta})$.

By contrast, Ornstein and Weiss showed that if $\Gamma$ is the rank
$2$ free group then $\Gamma\cc(\Z/2\Z)^{\Gamma}$ factors onto $\Gamma\cc(\Z/2\Z\times\Z/2\Z)^{\Gamma}$.
This convinced many researchers that there could not be an entropy
theory for free groups.

The factor map is defined by \[
\phi:(\Z/2\Z)^{\Gamma}\to(\Z/2\Z\times\Z/2\Z)^{\Gamma},\]
 \[
\phi(x)(g)=(x(g)+x(gs_{1}),x(g)+x(gs_{2})),\forall x\in(\Z/2\Z)^{\Gamma},g\in\Gamma.\]
 We consider $(\Z/2\Z)^{\Gamma}$ and $(\Z/2\Z\times\Z/2\Z)^{\Gamma}$
as compact groups under pointwise addition. It is a straightforward
exercise to show that $\phi$ is a surjective homomorphism which is
equivariant with respect to the shift-actions of $\Gamma$ and therefore,
defines a factor map. Moreover, the kernel of $\phi$ consists of
two elements, $x_{0},x_{1}$, where $x_{i}:\Gamma\to\Z/2\Z$ is defined
by $x_{i}(g)=i$. Let $N=\{x_{0},x_{1}\}$. Because $N$ is finite,
it clearly satisfies the conditions of Theorem \ref{thm:main}. So
that result implies \[
f(\Gamma\cc(\Z/2\Z)^{\Gamma})=f(\Gamma\cc N)+f(\Gamma\cc(\Z/2\Z\times\Z/2\Z)^{\Gamma}).\]
 In \cite{Bo10a}, it is shown that $f(\Gamma\cc(\Z/2\Z)^{\Gamma})=\log(2)$
and $f(\Gamma\cc(\Z/2\Z\times\Z/2\Z)^{\Gamma})=\log(4)$ as expected.
Therefore, $f(\Gamma\cc N)=-\log(2)$. This is easy to verify by direct
computation.

%
{}

\subsection{A generalization}

The example above can be generalized with the help of \cite[proof of Theorem B]{MRV11}
which states the following: if $\Gamma=\langle s_{1},\ldots,s_{r}\rangle$
is any finite rank free group, $K$ is any compact second countable
group, $K^{\Gamma}$ is the group of all functions $x:\Gamma\to K$
under pointwise multiplication and $K$ is identified with the constant
functions in $K^{\Gamma}$ then the action $\Gamma\cc K^{\Gamma}/K$
is measurably conjugate to $\Gamma\cc(K^{r})^{\Gamma}$ (where the
measures involved are the Haar measures and the actions are the shift
actions).


When $K$ is finite, we can apply Theorem \ref{thm:main} to obtain
\[
f(\Gamma\cc K^{\Gamma})=f(\Gamma\cc K)+f(\Gamma\cc(K^{r})^{\Gamma}).\]
 This is easy to check: $f(\Gamma\cc K^{\Gamma})=\log(|K|)$ and $f(\Gamma\cc(K^{r})^{\Gamma})=r\log(|K|)$
by \cite{Bo10a}. By a straightforward computation, $f(\Gamma\cc K)=-(r-1)\log|K|$.

\subsection{An algebraic example}

As above, let $\Gamma=\langle s_{1},\ldots,s_{r}\rangle$ be a finite
rank free group. Let $p>1$ be a prime number and $h\in(\Z/p\Z)\Gamma$.
We consider $h$ as a function from $\Gamma$ to $\Z$ such that $h(s)=0$
for all but finitely many $s\in\Gamma$. Define the convolution operator $\phi_{h}:(\Z/p\Z)^{\Gamma}\to(\Z/p\Z)^{\Gamma}$
by \[
\phi_{h}(x)(g)=\sum_{s\in\Gamma}x(gs)h(s^{-1}),\quad\forall g\in\Gamma.\]
 This is a $\Gamma$-equivariant homomorphism. Let $X_{h,p}$ denote
the kernel of $\phi_{h}$. Let $X_{h,p}^{*}<X_{h,p}$ be the subgroup
consisting of all elements $x\in X_{h,p}$ with $x(e)=0$. This is
a finite-index normal clopen subgroup and $\{gX_{h,p}^{*}:~g\in X_{h,p}\}$
is a generating partition for the shift-action of $\Gamma$. Therefore,
we can apply Theorem \ref{thm:main} to obtain \[
f(\Gamma\cc(\Z/p\Z)^{\Gamma})=f(\Gamma\cc X_{h,p})+f(\Gamma\cc\phi_{h}((\Z/p\Z)^{\Gamma})).\]

\begin{thm}
\label{thm:onto} $\phi_{h}$ is onto if $h$ is nonzero.
\end{thm}
Therefore, \[
f(\Gamma\cc\phi_{h}((\Z/p\Z)^{\Gamma}))=f(\Gamma\cc(\Z/p\Z)^{\Gamma}).\]
 Thus $f(\Gamma\cc X_{h,p})=0$.

To prove Theorem \ref{thm:onto}, we need a little preparation.
\begin{defn}
Let $C_{r}$ be the Cayley graph of $\Gamma$. It has vertex set $\Gamma$
and edges $\{g,gs_{i}\}$ for all $g\in\Gamma$ and $1\le i\le r$.
Given a set $F\subset\Gamma$, the {\em induced subgraph} of $F$
is the subgraph $C_{r}(F)\subset C_{r}$ which has vertex set $F$
and contains every edge of $C_{r}$ which has both endpoints in $F$.
A subset $F\subset\Gamma$ is said to be {\em connected} if its
induced subgraph in $C_{r}$ is connected. The {\em convex hull}
of a set $F\subset\Gamma$ is the smallest connected set $F'\subset\Gamma$
with $F\subset F'$. An {\em extreme point} of $F$ is an element
$f\in F$ that has degree 1 in $C_{r}(F)$. We let $\Ex(F)$ denote
the set of extreme points of $F$. Note that if $F'$ is the convex
hull of $F$ then $\Ex(F')\subset F$. \end{defn}
\begin{lem}
\label{lem:order} Let $F=\{g\in\Gamma:~h(g^{-1})\ne p\Z\}$. Let
$\oF$ be the convex hull of $F$. Suppose there exists an ordering
$\gamma_{0},\gamma_{1},\gamma_{2},\ldots$ of $\Gamma$ such that
for every $n\ge1$ $\{\gamma_{0},\ldots,\gamma_{n}\}$ is connected
and \[
\gamma_{n}\oF\nsubseteq\cup_{i=0}^{n-1}\gamma_{i}\oF.\]
 Then $\phi_{h}$ is onto. \end{lem}
\begin{proof}
By compactness of $(\Z/p\Z)^{\Gamma}$ and continuity of $\phi_{h}$,
it suffices to show that for every $y\in(\Z/p\Z)^{\Gamma}$ and every
$n\ge0$, there exists an $x\in(\Z/p\Z)^{\Gamma}$ such that $\phi_{h}(x)(\gamma_{i})=y(\gamma_{i})$
for every $0\le i\le n$. We will prove this statement by induction
on $n$. It is clearly true for $n=0$. So suppose there is an $n\ge0$
for which the statement is true. Fix $y\in(\Z/p\Z)^{\Gamma}$ and
let $x\in(\Z/p\Z)^{\Gamma}$ be such that $\phi_{h}(x)(\gamma_{i})=y(\gamma_{i})$
for every $0\le i\le n$.

By hypothesis, $\gamma_{n+1}\oF\nsubseteq\cup_{i=0}^{n}\gamma_{n}\oF$.
Because $\cup_{i=0}^{n}\gamma_{n}\oF$ and $\gamma_{n+1}\oF$ are
connected and the convex hull of the exreme points set of a connected
set is the connected set itself , there must be an extremal point
$f\in\Ex(\oF)$ such that $\gamma_{n+1}f\notin\cup_{i=0}^{n}\gamma_{n}\oF$.
However, $\Ex(\oF)\subset F$. So $f\in F$. By definition, this means
that $h(f^{-1})\ne p\Z$. Because $p$ is prime, we may therefore
define an element $m\in\Z/p\Z$ by \[
m=h(f^{-1})^{-1}\left(y(\gamma_{n+1})-\sum_{g\in\Gamma\setminus\{f\}}x(\gamma_{n+1}g)h(g^{-1})\right).\]
 Define $x'\in(\Z/p\Z)^{\Gamma}$ by $x'(g)=x(g)$ if $g\ne\gamma_{n+1}f$
and $x'(\gamma_{n+1}f)=m$. Because $\gamma_{n+1}f\notin\cup_{i=0}^{n}\gamma_{n}\oF$,
it follows that $\phi_{h}(x')(\gamma_{i})=\phi_{h}(x)(\gamma_{i})$
for all $0\le i\le n$. Also a straightforward computation shows $\phi_{h}(x')(\gamma_{n+1})=y(\gamma_{n+1})$.
So $\phi_{h}(x')(\gamma_{i})=y(\gamma_{i})$ for all $0\le i\le n+1$.
This completes the inductive step and the claim. \end{proof}
\begin{defn}
Let $S=\{s_{1},\ldots,s_{r}\}$. For $g\in\Gamma$, let $|g|$ be
the smallest number $n\ge0$ such that there exist elements $t_{1},\ldots,t_{n}\in S\cup S^{-1}$
with $g=t_{1}\cdots t_{n}$. We also let $d(g_{1},g_{2})=|g_{1}^{-1}g_{2}|$
for any $g_{1},g_{2}\in\Gamma$. For $g\in\Gamma$ and $n\ge0$, let
$B(g,n)=\{k\in\Gamma:~d(k,g)\le n\}$ be the ball of radius $n$ centered
at $g$. 

Let $K\subset\Gamma$ be a finite set. The {\em radius} of $K$
is the smallest number $r\ge0$ such that there exists a $v\in\Gamma$
such that $B(v,r)\supset K$. An element $v\in\Gamma$ is called a
{\em center} of $K$ if $B(v,r)\supset K$ where $r$ is the radius
of $K$. For any $v,w\in\Gamma$, we let $[v,w]\subset\Gamma$ be
the set of all $g\in\Gamma$ such that the shortest path from $v$
to $w$ in the Cayley graph $C_{r}$ contains $g$. \end{defn}
\begin{lem}
Let $K$ be a connected finite set with radius $r\ge1$. Suppose the
identity element $e$ is a center of $K$. Then there exist elements
$v,w\in K$ such that $[e,v]\cap[e,w]=\{e\}$, $|v|=r$ and $|w|\in\{r-1,r\}$. \end{lem}
\begin{proof}
Because $K$ has radius $r$ and center $e$, there is an element
$v$ with $|v|=r$. To obtain a contradiction, suppose that there
is no $w\in K$ with $|w|\in\{r-1,r\}$ and $[e,v]\cap[e,w]=\{e\}$.
Let $v_{1}\in S\cup S^{-1}$ be the unique element with $|v_{1}^{-1}v|=r-1$.
We claim that $B(v_{1},r-1)\supset K$. To see this, let $w\in K$.
If $|w|\le r-2$ then $w\in B(e,r-2)\subset B(v_{1},r-1)$. If $|w|>r-2$
then, because $K$ has center $e$ and radius $r$, $|w|\in\{r-1,r\}$.
By assumption, this implies $[e,v]\cap[e,w]\ne\{e\}$. So let $y\in[e,v]\cap[e,w]$
with $y\ne e$. Then $[e,y]\subset[e,v]$. This implies that $v_{1}\in[e,y]$.
In particular, $v_{1}\in[e,v]\cap[e,w]$, so $v_{1}\in[e,w]$. Because
$|w|\le r$, this implies $d(v_{1},w)\le r-1$ as claimed. So we have
shown that in all cases, if $w\in K$ then $w\in B(v_{1},r-1)$. This
shows that the radius of $K$ is at most $r-1$, a contradiction.
This contradiction proves the lemma. \end{proof}
\begin{lem}
Let $K$ be a connected finite set with radius $r\ge1$. Suppose the
identity element $e$ is a center of $K$. Suppose $g_{1},\ldots,g_{n}\in\Gamma\setminus\{e\}$
are elements with \[
K\subset\cup_{i=1}^{n}g_{n}K.\]
 Then $e$ is contained in the convex hull of $\{g_{1},\ldots,g_{n}\}$. \end{lem}
\begin{proof}
Let $v,w\in K$ be elements such that $[e,v]\cap[e,w]=\{e\}$, $|v|=r$
and $|w|\in\{r-1,r\}$. Let $g_{i},g_{j}\in\{g_{1},\ldots,g_{n}\}$
be such that $v\in g_{i}K$ and $w\in g_{j}K$. Let $x,y\in K$ be
such that $v=g_{i}x$ and $w=g_{j}y$.

Let $v_{1},v_{2},x_{1},x_{2}\in\Gamma$ be such that $v=v_{1}v_{2},|v|=|v_{1}|+|v_{2}|$,
$x_{2}=v_{2}$, $x=x_{1}x_{2}$, $|x|=|x_{1}|+|x_{2}|$ and $|v_{2}|=|x_{2}|$
is as large as possible. Thus $g_{i}=vx^{-1}=v_{1}x_{1}^{-1}$ and
$|vx^{-1}|=|v_{1}|+|x_{1}|$. Because $r$ is the radius of $K$,
$e$ is a center and $x\in K$ we have $|x|\le r$. Also, we cannot
have $v=x$ (since this would imply $g_{i}=vx^{-1}=e$, a contradiction).
So we must have $|v_{1}|\ge1$. Thus $[e,v]\cap[e,g_{i}]\ne\{e\}$.

Let $w_{1},w_{2},y_{1},y_{2}\in\Gamma$ be such that $w=w_{1}w_{2},|w|=|w_{1}|+|w_{2}|$,
$y_{2}=w_{2}$, $y=y_{1}y_{2}$, $|y|=|y_{1}|+|y_{2}|$ and $|w_{2}|=|y_{2}|$
is as large as possible. Thus $g_{j}=wy^{-1}=w_{1}y_{1}^{-1}$ and
$|wy^{-1}|=|w_{1}|+|y_{1}|$. Because $r$ is the radius of $K$,
$e$ is a center and $y\in K$ we have $|y|\le r$.

\textbf{Case 1}. If $|w|=r$, then, as in the previous paragraph,
we must have $[e,w]\cap[e,g_{j}]\ne\{e\}$. Because $[e,v]\cap[e,w]=\{e\}$,
this implies $e\in[g_{i},g_{j}]$ which implies the lemma.


\textbf{Case 2}. Suppose $|w|=r-1$ and $|w_{1}|\ge1$. Thus $[e,w]\cap[e,g_{j}]\ne\{e\}$.
Because $[e,v]\cap[e,w]=\{e\}$, this implies $e\in[g_{i},g_{j}]$
which implies the lemma.

\textbf{Case 3}. Suppose $|w|=r-1$ and $|w_{1}|=0$. Then $w=w_{2}$,
so $|w_{2}|=r-1$. Because $g_{j}=wy^{-1}=w_{1}y_{1}^{-1}=y_{1}^{-1}\ne e$,
we must $y_{1}\ne e$. Thus $|y|=|y_{1}|+|y_{2}|=|y_{1}|+|w_{2}|=|y_{1}|+r-1$.
Because $y\in K$ and $K$ has radius $r$ and center $e$, we must
have $|y_{1}|=1$ and $|y|=r$. If $[e,y]\cap[e,v]=\{e\}$ then, after
replacing $w$ with $y$ we are in the situation of Case 1 (note $y=g_{k}y'$
for some $1\leq k\leq n$ and $y'\in K$). So we may assume $[e,y]\cap[e,v]\ne\{e\}$
which implies $y_{1}\in[e,v]$. Because $g_{j}=y_{1}^{-1}$, and $[e,v]\cap[e,g_{i}]\ne\{e\}$,
we have $[e,g_{i}]\cap[e,g_{j}]=\{e\}$ which implies $e\in[g_{i},g_{j}]$
which implies the lemma.

{[}Proof of Theorem \ref{thm:onto}{]} Let $F=\{g\in\Gamma:~h(g^{-1})\ne p\Z\}$.
Let $\oF$ be the convex hull of $F$. For any $g\in\Gamma$, $\phi_{h}$
is onto if and only if $\phi_{gh}$ is onto. So after replacing $h$
with $gh$ for some $g\in\Gamma$, we may assume that $e$ is a center
of $\oF$.

Let $g_{0},g_{1},\ldots$ be an ordering of $\Gamma$ such that for
every $n\ge0$, $\{g_{0},\ldots,g_{n}\}$ is connected. We claim that
for every $n\ge1$, \[
\gamma_{n}\oF\nsubseteq\cup_{i=0}^{n-1}\gamma_{i}\oF.\]
 To obtain a contradiction, suppose that the claim is false for some
$n\ge1$. Then $\oF\subset\cup_{i=0}^{n-1}\gamma_{n}^{-1}\gamma_{i}\oF$,
$\gamma_{n}^{-1}\gamma_{i}\ne e$ for any $0\le i\le n-1$ and because
$\{\gamma_{0},\ldots,\gamma_{n-1}\}$ is connected, $\{\gamma_{n}^{-1}\gamma_{0},\ldots,\gamma_{n}^{-1}\gamma_{n-1}\}$
is connected which implies that $e$ is not in the convex hull of
$\{\gamma_{n}^{-1}\gamma_{0},\ldots,\gamma_{n}^{-1}\gamma_{n-1}\}$.
This contradicts the previous lemma. So we must have that for every
$n\ge1$, \[
\gamma_{n}\oF\nsubseteq\cup_{i=0}^{n-1}\gamma_{i}\oF.\]
 The theorem now follows from Lemma \ref{lem:order}.
\end{proof}


\vspace{0.5cm}

\address{Lewis Bowen, Mathematics Department, Mailstop 3368, Texas A\&M University
College Station, TX 77843-3368 United States.}

\textit{E-mail address}: \texttt{lpbowen@math.tamu.edu}

\vspace{0.5cm}

\address{Yonatan Gutman, Institute of Mathematics, Polish Academy of Sciences,
ul. \mbox{\'{S}niadeckich 8}, 00-956 Warszawa, Poland.}

\textit{E-mail address}: \texttt{y.gutman@impan.pl}
\end{document}